\documentclass[preprint,12pt]{elsarticle}

\usepackage{array}
\usepackage[hidelinks]{hyperref}
\usepackage{amsmath}
\usepackage{amsfonts}
\usepackage{amsthm}
\usepackage{calrsfs}
\usepackage{enumerate}

\usepackage{amssymb,latexsym}

\newcommand{\FF}{\mathbb F}

\newcommand{\cL}{\mathcal L}

\newcommand{\cP}{\mathcal P}

\newcommand{\cC}{\mathcal C}

\newcommand{\PG}{\mathrm{PG}}
\newcommand{\vep}{\varepsilon }
\newcommand{\PGL}{\mathrm{PGL}}

\newcommand{\bee}[1]{\mathbf{e}_{#1}}
\newcommand{\GL}{\mathrm{GL}}

\newcommand{\Aut}{\mathrm{Aut}\,}

\newcommand{\Tr}{\mathrm{Tr}}

\newcommand{\rank}{\mathrm{rank}\,}

\newtheorem{theorem}{Theorem}[section]
\newtheorem{lemma}[theorem]{Lemma}
\newtheorem{corollary}[theorem]{Corollary}
\newtheorem{prop}[theorem]{Proposition}

\theoremstyle{definition}
\newtheorem{remark}{Remark}

\newtheorem{definition}[theorem]{Definition}
\newcommand{\cH}{\mathcal H}
\newcommand{\cM}{\mathcal M}

\begin{document}
\begin{frontmatter}
  \title{Linear codes arising from the point-hyperplane geometry-Part I: the
    Segre embedding}
    \author{I.~Cardinali}
    \ead{ilaria.cardinali@unisi.it}
    \affiliation{
      organization={Dep. Information Engineering and Mathematics, University of Siena},
      addressline={Via Roma 56},
      city={Siena},
      postcode={I-53100},
      country={Italy}}

    \author{L.~Giuzzi}
    \ead{luca.giuzzi@unibs.it}
    \affiliation{
      organization={DICATAM, Universit\`a di Brescia},
      addressline={Via Branze 43},
      city={Brescia},
      postcode={I-25123},
      country={Italy}}
\begin{abstract}
  Let $V$ be a vector space over the finite field $\FF_q$
  with $q$ elements and $\Lambda$ be the image of the Segre geometry $\PG(V)\otimes\PG(V^*)$ in $\PG(V\otimes V^*)$.
  Consider the subvariety $\Lambda_{1}$ of $\Lambda$ represented by the pure tensors $x\otimes \xi$ with $x\in V$ and $\xi\in V^*$ such that $\xi(x)=0$.
  Regarding $\Lambda_1$ 
as a projective system of $\PG(V\otimes V^*)$, we study the linear code $\cC(\Lambda_1)$ 
arising from it.
We show that
$\cC(\Lambda_1)$ is a minimal code and we determine its basic parameters,
its  full weight list and its linear automorphism group.  We also give a geometrical characterization of its minimum and second lowest weight codewords as well as of some of the words of maximum weight.
\end{abstract}

\end{frontmatter}

\section{Introduction}
Following~\cite{TVZ}, an interesting and very geometric way to construct linear codes is to start from a \emph{projective system}, that is a spanning set of points in a finite projective space.
Indeed, if $\Omega $ is a set of $N$ distinct points of an $n$-dimensional projective space over a finite field $\FF_q$ then it is possible to construct a projective code  $\cC(\Omega)$ out of $\Omega$
by taking the coordinate representatives (with respect to some fixed reference
system) of the points of $\Omega$
as columns of a generator matrix.

In general, $\cC(\Omega)$ is not uniquely determined by $\Omega$, but it turns out to
be unique up to monomial equivalence; as such its metric properties
with respect to
Hamming's distance depend only on the set of points under consideration.
With a slight abuse of notation, which is however customary when dealing
with such projective codes, we shall speak of $\cC(\Omega)$ as \emph{the} code defined
by $\Omega=\{ [v_1],[v_2],\dots, [v_N]\}$ where the $v_i$'s  are fixed
chosen vector representatives of the points of the projective system.

The parameters $[N,k,d]$ of $\cC(\Omega)$ depend only on  the point set
$\Omega$: clearly, the length $N$ is the size of $\Omega$ and the dimension $k$
is the (vector) dimension of the subspace spanned
by $\Omega$. The spectrum of the cardinalities of the intersections of $\Omega$
with the hyperplanes of $\PG(\langle\Omega\rangle)$ is used to
determine the list of the weights of the code;
in particular, the minimum Hamming distance $d$ of $\cC(\Omega)$ is
\begin{equation}\label{min distance 1}
	d =N-\max_{H}|\Omega\cap H|,
\end{equation}
as $H$ ranges among all hyperplanes of the space $\PG(\langle\Omega\rangle)$;
we refer to~\cite{TVZ} for further details.

This approach for obtaining codes from projective systems has been effectively applied to construct and study
several families of projective codes, such as, for example,
Hermitian codes \cite{Bonini2023} 
Grassmann codes \cite{Ghorpade2001,Nogin1996},  Schubert codes
\cite{Ghorpade2013,Ghorpade2005},  polar Grassmann codes
\cite{Cardinali2016b, Cardinali2018, Cardinali2018c, Cardinali2016a}.

In this paper we shall consider one family of projective codes arising from a special subvariety of a Segre variety corresponding to the Segre embedding of the point-hyperplane geometry of $\PG(V)$.

Suppose that $V$ is an $(n+1)$-dimensional vector space over a finite field $\FF_q$ and let $V^*$ be its dual. The tensor product $V\otimes V^*$  can be identified with the vector space $M_{n+1}(q)$ of all square matrices of
order $n+1$ with coefficients in $\FF_q$.

The Segre variety $\Lambda$ of $\PG(V\otimes V^*)\cong\PG(M_{n+1}(q))$ is the variety whose point set consists of those projective points  represented by pure tensors $x\otimes \xi$ with $x\in V$ and $\xi\in V^*$, i.e. by $(n+1)$-square matrices of rank $1$ in $\PG(M_{n+1}(q))$. Another way of defining $\Lambda$ is to regard it as the image under the Segre map of the Segre geometry
$\Gamma=\PG(V)\otimes\PG(V^*)$; see Section~\ref{Segre geometry} for more information.

The
linear code $\cC(\Lambda)$ arising from $\Lambda$ has been introduced and extensively studied in~\cite{BGH}. In particular, not only the minimum distance but also the full Hamming weight distribution of $\cC(\Lambda)$ are known (see~\cite[§ 4]{BGH}), as well as its higher Hamming weights.

Here,  we  investigate the subvariety $\Lambda_{1}$ of the Segre variety  whose point set  is represented by $(n+1)$-square matrices of rank $1$ having null trace, or equivalently, defined  by the pure tensors $x\otimes \xi$ with $x\in V$ and $\xi\in V^*$ such that $\xi(x)=0$.

As mentioned before, the variety $\Lambda_1$ can also be regarded as the image under the Segre embedding of
the point-hyperplane geometry of $\PG(V)$ (also called the {\it long root geometry} $\bar{\Gamma}$ for the special linear group $\mathrm{SL}(n+1,q)$).
Briefly, the points of $\bar{\Gamma}$ are the point-hyperplane pairs  $(p,H)$ of $\PG(V)$ where $p\in H$.  We refer to Section~\ref{long root geometry} for more information on $\bar{\Gamma}$.

We focus on the linear code $\cC(\Lambda_1)$ constructed as explained at the beginning of the Introduction,
starting from the projective system $\Lambda_1$. We determine the parameters of  $\cC(\Lambda_1)$ and its full weight list, as well as a description of its linear automorphism group; see~\cite{MS} for the definitions.
We prove that of $\cC(\Lambda_1)$ is a {\it minimal code} and we give a complete geometric characterization of the words having  minimum weight and second lowest weight. We also characterize a remarkable family of words
of maximum weight.

The following are the main results of the paper.
\begin{theorem} \label{main thm 1} Suppose $V$ is an $(n+1)$-dimensional vector space over $\FF_q$ and let $\Lambda_{1}$ be the projective system of $\PG(V\otimes V^*)$ whose points are represented by the pure tensors $x\otimes \xi$ such that $\xi(x)=0$. The $[N_1,k_1,d_1]$-linear code $\cC(\Lambda_1)$ associated to $\Lambda_1$ has parameters \[N_1=\frac{(q^{n+1}-1)(q^n-1)}{(q-1)^2},
    \qquad k_1=n^2+2n,\qquad d_1=q^{2n-1} -q^{n-1}.\]
\end{theorem}
A direct consequence of Theorem~\ref{main thm 1} is that
$\lim_{n\to\infty}\frac{d_1}{N_1}=1$; so, the asymptotic
minimum distance for these codes is actually good, even if the information
rate $\lim_{n\to\infty}\frac{k_1}{N_1}$ goes to $0$.

The following theorem gives information on the weight list of $\cC(\Lambda_1)$
as well as its automorphism group.
\begin{theorem} \label{main thm 2}
		Let $\cC(\Lambda_1)$ be the linear code introduced in Theorem~\ref{main thm 1}. The following hold.
	\begin{enumerate}
        \item\label{t-cor}
          The set of weights of $\cC(\Lambda_1)$ is in bijective correspondence with the set
\[\{(g_1,\dots, g_t)\colon \sum_{i=1}^tg_i \leq n+1,\,\, 1\leq g_1\leq \dots \leq  g_t\leq  n+1,\,\,1 \leq t \leq q\}\cup \{0\}.\]
\item\label{t-isom}
  The code admits an automorphism group isomorphic to
  the central product $\PGL(n+1,q)\cdot\FF_q^{\star}$.
\end{enumerate}
\end{theorem}

Many properties of the geometry $\bar{\Gamma}$  will play a crucial role in giving information on the structure of the code
$\cC(\Lambda_1)$ arising from $\Lambda_{1} $.

In particular, the concept of {\it geometrical hyperplane of $\bar{\Gamma}$ arising from an embedding}\/ recalled in Section~\ref{prelim} will be used to characterize codewords  having minimum, second lowest or maximum weight in geometric terms. We refer to Section~\ref{forma saturata} for the definition and description of the hyperplanes of $\bar{\Gamma}$ mentioned in Theorem~\ref{main thm 2-bis}.

\begin{theorem} \label{main thm 2-bis}
	Let $\cC(\Lambda_1)$ be the linear code introduced in Theorem~\ref{main thm 1}. The following hold.
  \begin{enumerate}
  \item \label{mw-1} $\cC(\Lambda_1)$ is a minimal code.
  \item\label{mw0}
    The minimum weight codewords of $\cC(\Lambda_1)$ have weight
     $q^{2n-1}-q^{n-1}$ and correspond to
     the   quasi-singular but non-singular hyperplanes of the point-hyperplane
     geometry $\bar{\Gamma}$.
     These hyperplanes arise from
    (diagonalizable) matrices of rank $1$ and non-null trace.

  The second lowest weight codewords of $\cC(\Lambda_1)$ have weight $q^{2n-1}$
    and correspond to the singular hyperplanes of
    $\bar{\Gamma}$. They arise from (non-diagonalizable) matrices
    of rank $1$ and null trace.
  \item\label{maxw0} The maximum weight codewords of $\cC(\Lambda_1)$ have weight
  $q^{n-1}(q^{n+1}-1)/(q-1)$ and correspond to matrices admitting no eigenvalue in $\FF_q$. Every spread-type hyperplane of $\bar{\Gamma}$ is associated to a maximum weight codeword. Conversely, those maximum weight codewords arising from matrices having  a minimal polynomial of degree $2$ are associated to
  {\it spread-type hyperplanes}\/ of $\bar{\Gamma}$.

  \end{enumerate}
\end{theorem}

\subsection{Organization of the paper}
In Section~\ref{Sec 1} we will set the notation (Subsection~\ref{notation}) and recall all the basics we need about embeddings of geometries (Subection~\ref{prelim}),
minimal codes (Subsection~\ref{ssmin}),
the Segre geometry (Subsection~\ref{Segre geometry}), the point-hyperplane geometry $\bar{\Gamma}$ of
$\PG(V)$, its embeddings (Subsection~\ref{long root geometry}) and
the saturation form (Subsection~\ref{forma saturata}).
We will then define the quasi-singular, singular and spread-type hyperplanes of $\bar{\Gamma}$ (Subsection~\ref{hyperplanes}).

Section~\ref{long root geometry code 1} is focused on the code arising from the variety $\Lambda_1$ of the pure tensors $x\otimes \xi$ with $x\in V$ and $\xi\in V^*$ such that $\xi(x)=0$ and here we will prove  Theorems~\ref{main thm 1}, \ref{main thm 2} and \ref{main thm 2-bis}.

\section{Notation and Basics}\label{Sec 1}
\subsection{Notation}\label{notation}
Let $V=V(n+1,\FF_q)$ be an $(n+1)$-dimensional vector space over the finite field $\FF_q$ and $V^*$ its dual. Henceforth we always assume that $E = (e_i)_{i=1}^{n+1}$ is a given basis of $V$
and $E^*=(\eta_i)_{i=1}^{n+1}$ is the dual of $E$; hence a basis of $V^*$.
We shall regard a vector $x = \sum _{i=1}^{n+1}e_i x_i $  of $V$,
represented by the $(n + 1)$-tuple $(x_i)_{i=1}^{n+1}$
of its coordinates with respect to $E$, as a column, namely an $(n+1) \times 1$-matrix.

Similarly, every vector $\xi = \sum _{i=1}^{n+1} {\xi}_i {\eta}_i \in V^* $ is regarded with
respect to $E^*$ as a $1 \times (n+1)$ matrix $(\xi_1, \xi_2, \dots, \xi_{n+1})$.
Clearly, $\xi x=0$ in terms of the usual row-by-column product if and
only if $\xi(x)=0$, that is
$\sum_{i=1}^{n+1}\xi_i x_i=0$.
Henceforth we shall always use Greek lower case letters to denote vectors
of $V^*$ (regarded as \emph{row} vectors) and Roman lower case letters
to denote vectors of $V$ (regarded as column vectors).

The tensor product $V\otimes V^*$ is  isomorphic to the vector space $M_{n+1}(q)$ of the square matrices of order $n+1$ with entries in $\FF_q$. Hence, we will freely switch
from the matrix notation to the tensor notation and conversely, whenever these
changes of notation will be convenient.
The elements $E\otimes E^*=\{e_i\otimes \eta_j\}_{1\leq i,j\leq n+1}$ form a basis of $V\otimes V^*$. If we map $e_i\otimes\eta_j$ to the elementary matrix $\bee{ij}$ whose only non-zero entry is $1$ in position $(i,j)$, we
see that this gives the isomorphism $\phi\colon V\otimes V^* \rightarrow M_{n+1}(q)$ described by
\[ x\otimes\xi \in V\otimes V^*\mapsto \begin{pmatrix} x_1\\\vdots\\x_{n+1} \end{pmatrix}
  \begin{pmatrix} \xi_1 & \dots & \xi_{n+1} \end{pmatrix}\in M_{n+1}(q),\]
where
\[ x=\begin{pmatrix} x_1 \\ \vdots \\ x_{n+1}\end{pmatrix}\qquad
  \text{ and }\qquad \xi=\begin{pmatrix}
    \xi_1 & \dots & \xi_{n+1} \end{pmatrix}. \]
Thus, the tensor $x \otimes \xi\in V\otimes V^*$ can be
regarded as the column-times-row
product $x \xi$ and  for a matrix $M \in M_{n+1}(q)$, the product $\xi M x$ is the scalar obtained as the product of the row $\xi$ times
$M$ times the column $x$.

Turning to projective spaces, let $\PG(n,q)=\PG(V)$ be the $n$-dimensional projective space defined by $V$.
When we need to distinguish between a non-zero vector $x$ of $V$ and the point
of $\PG(V )$ represented by it, we denote the latter by $[x]$. We extend
this convention to subsets of $V$. If $X \subseteq V \setminus \{0\}$ then $[X] := \{[x] | x \in X\}$. The
same conventions will be adopted for vectors and subsets of $V^*$ and $V \otimes V^*$. In
particular, if $\xi \in V^* \setminus \{0\}$ then $[\xi]$ is the point of $\PG(V^*)$ which corresponds
to the hyperplane $[\ker(\xi)]$ of $\PG(V )$. In the sequel we shall freely take $[\xi]$ as a
name for $[\ker(\xi)]$. Accordingly, if $0^*\in V^*$ is
the null functional, we write $[0^*]=\PG(V)$ (or, more simply, $[0]=\PG(V)$).

 \subsection{Embeddings and hyperplanes of point-line geometries.}\label{prelim}
The approach we follow to construct our codes is to start from a point-line geometry and then embed it into a projective space, so that its image will be a spanning set, hence a projective system, for the ambient projective space.

More precisely, let $\Gamma=(\mathcal{P},\mathcal{L})$ be a point-line geometry with
point set $\mathcal{P}$, line-set $\mathcal{L}$ and incidence given by inclusion. The collinearity graph $G_{\Gamma}$ of $\Gamma$ is
the graph whose vertices are the points $p\in\cP$  and whose edges
are the pairs $(p,q)\in\cP\times\cP$ where $p$ is collinear with $q$.

A \emph{subspace} of a point-line geometry $\Gamma$ is a non-empty subset $X$ of the point set $\mathcal{P}$ of $\Gamma$ such
that, for every line $\ell$ of $\Gamma$, if $|\ell \cap X| > 1$ then $\ell \subseteq X$. A proper subspace $H$ of $\Gamma$ is
said to be a {\it geometric hyperplane}\/ of $\Gamma$ (a \emph{hyperplane} of $\Gamma$ for short) if
for every line $\ell$ of
$\Gamma$, either $\ell \subseteq  H$ or $|\ell  \cap H | = 1$.
By definition, hyperplanes are proper subspaces of $\Gamma$ but in general
they are not necessarily all \emph{maximal} with respect to
inclusion among all the proper subspaces.
An important characterization of \emph{maximal hyperplanes}
is the following, see~\cite[Lemma 4.1.1]{Shult2011}.
\begin{prop}
  \label{cmhp}
  A geometric hyperplane $H$ of a point-line geometry $\Gamma$ is
  a maximal subspace if and only if the collinearity
  graph induced by $G_{\Gamma}$
  on $\cP\setminus H$ is connected.
\end{prop}
For more details we refer the reader to~\cite{Shult1995} and
\cite{Shult2011}.

A \emph{projective embedding} $\vep\colon \Gamma \rightarrow \Sigma$ in a projective space $\Sigma$ is an injective mapping $\varepsilon$ from the point set  $\mathcal{P}$ of $\Gamma$  to the point set of a projective space $\Sigma$  such that $\varepsilon(\mathcal{P})$ spans $\Sigma$ and it maps lines of $\Gamma$ onto projective lines of $\Sigma$. The \emph{dimension of} $\varepsilon$ is the vector dimension of $\Sigma$.

Given a projective embedding $\vep\colon \Gamma \rightarrow \Sigma$
of $\Gamma$ and a projective hyperplane $W$ of $\Sigma$, the point set $\mathcal{W}:=\vep^{-1}(W)\subseteq\Gamma$ is a geometric hyperplane of $\Gamma$ and $\vep( \vep^{-1}(W)) = W\cap  \vep(\mathcal{P})$.

We say that a geometric hyperplane $\mathcal{W}$ of $\Gamma$
\emph{arises from $\vep$} if $\vep(\mathcal{W})$ spans a projective hyperplane of $\Sigma$ and $\mathcal{W}=\vep^{-1}(W)\subseteq\Gamma$ for some hyperplane $W$ of $\Sigma$.
Note that, in general, it is possible that there are many hyperplanes of $\Gamma$ which do not
arise from any embedding.
\subsection{Minimal codewords and minimal codes.}\label{ssmin}
Let $\cC=\cC(\Omega)$ be a projective $[N,k,d]$-code, where $\Omega\subseteq \PG(k-1,q)$ is the defining projective system.
For any $c=(c_1,\dots,c_N)\in\cC$ the \emph{support} of $c$ is the set
$\mathrm{supp}(c)=\{i: c_i\neq 0\}$.
Write also $\mathrm{supp}^c(c):=\{i: c_i=0\}=\{1,\dots,N\}\setminus
\mathrm{supp}(c)$.

Massey in 1993 \cite{Massey} introduced the notion of {\it minimal codewords}
(or minimal vectors)
in a code,
in order to devise an efficient secret sharing scheme.
Minimal codewords have also numerous applications besides
cryptography; e.g.\ they
are relevant for bounding the complexity of
some decoding algorithms; see~\cite{Ashikhmin1998}.

Minimal codewords and minimal codes are defined as follows.

\begin{definition}
  \label{minc}
  A codeword $c\in\cC$ is \emph{minimal} if
  \[ \forall c'\in\cC:\mathrm{supp}(c')\subseteq\mathrm{supp}(c)
    \Rightarrow\exists\lambda\in\FF_q: c'=\lambda c. \]
  A code is \emph{minimal} if all its codewords are minimal.
\end{definition}

Minimal codes have been extensively investigated,
since they are also amenable to efficient decoding~\cite{Ashikhmin1998},
see also~\cite{Cohen2013,Alfarano2022,
  Alon2024,Bonini2021, survey}.

Ashikhmin and Barg in 1998~\cite{Ashikhmin1998} determined
a well-known and widely used
sufficient condition for a code $\cC$ to be minimal:
\begin{equation}
  \label{ab-bound}
  \frac{w_{\max}}{w_{\min}}<\frac{q}{q-1},
\end{equation}
where $w_{\min}$ and $w_{\max}$ are respectively the minimum and
the maximum weight of the non-null codewords of $\cC$. The aim is
to determine codes which are minimal but do not satisfy
condition~\eqref{ab-bound}; see e.g. \cite{Bonini2021}.

The notion of \emph{strong} or \emph{cutting set} has been introduced in~\cite{D2011}.
\begin{definition}
  Let $\Omega\subseteq\PG(\langle\Omega\rangle)$ be a projective system.
  Then $\Omega$ is a \emph{cutting set} (with respect to the hyperplanes)
  if and only if for any hyperplane $H$ of $\PG(\langle\Omega\rangle)$,
  \[ \langle H\cap\Omega\rangle=H. \]
\end{definition}
In 2021 Alfarano et al.~\cite{Alfarano2022a}
 proved that projective minimal codes and cutting sets with respect
to hyperplanes are equivalent objects (see also~\cite{Tang2021}).

Relying on the notion of cutting sets and on Proposition~\ref{cmhp}, we have proved in~\cite{survey} the following.
\begin{prop}
  Let $\Gamma:=(\cP,\cL)$ be a point line geometry and  $\vep:\Gamma\to\PG(V)$
  a projective embedding.
  The projective code $\cC(\vep(\Gamma))$ is minimal if and only if
  for any hyperplane $H$ of $\PG(V)$,
  the graph induced on the vertices $\cP\setminus\vep^{-1}(H)$ by
  the collinearity graph of $\Gamma$ is connected.
\end{prop}

 \begin{prop}
  \label{c:min}
  Suppose that $\Gamma=(\cP,\cL)$ is a point-line geometry where every
  geometric hyperplane is a maximal subspace.
  Then the projective code
  $\cC(\vep(\Gamma))$ is minimal, for any projective embedding
  $\vep$ of $\Gamma$.
\end{prop}

\subsection{The Segre geometry and its natural embedding.}\label{Segre geometry}
Suppose $\PG(V_1)$ and $\PG(V_2)$ are two given projective spaces
having respective point sets $\cP_1$ and $\cP_2$ and
respective line sets
 $\cL_1$ and $\cL_2$.
The \emph{Segre geometry} (of type $(\dim(\PG(V_1)),\dim(\PG(V_2)))$ ) is defined as the point-line geometry  $\Gamma$ having as point set the Cartesian product $\cP_1\times\cP_2$ and as line set the following set:
\[\{ \{p_1\}\times\ell_2\colon p_1\in\cP_1,\ell_2\in\cL_2\}
  \cup
   \{\ell_1\times \{p_2\}\colon \ell_1\in\cL_1, p_2\in\cP_2\};
 \]
 incidence is given by inclusion.
 It is well known that  this geometry, which can be regarded as the
 product $\PG(V_1)\otimes\PG(V_2),$ admits a projective embedding (the Segre embedding) in $\PG(V_1\otimes V_2)$ mapping the point  $([p],[q])\in\cP_1\times\cP_2$ to the projective point
 $[p\otimes q]\in \PG(V_1\otimes V_2)$. This embedding lifts the automorphism group $\PGL(V_1)\times\PGL(V_2)$ of $\Gamma$ to act on $\PG(V_1\otimes V_2)$. The geometric hyperplanes  of the Segre geometries have been fully classified and described;  see~\cite{Hvm24}.

 In this paper we shall be concerned with the case where
 $V_1=V(n+1,\FF_q)$ and $V_2=V_1^*$. Thus, the points of $\Gamma$ are all the  ordered pairs $([p],[\xi])$ where $[p]$ and $[\xi]$ are respectively a point and a hyperplane of $\PG(n,q).$
We will denote by $\vep$ the Segre embedding of $\Gamma$ in $\PG(V\otimes V^*)$, also called the \emph{natural embedding of $\Gamma$.}  Since $V\otimes V^*$ is isomorphic to the vector space $M_{n+1}(q)$ of the square matrices of order $n+1$ with entries in $\FF_q$ (see Subection~\ref{notation}), the Segre embedding is given by
\begin{equation}
  \label{segre embedding}
  \varepsilon \colon \Gamma \rightarrow \PG(M_{n+1}(q)),\,\,\varepsilon (([x], [\xi]))=[x\otimes \xi].
\end{equation}
The linear automorphism group $\PGL(V)\otimes\PGL(V^*)$
of $\Gamma$ acts on the pure tensors of $\PG(M_{n+1}(q))$ as
\[ ([M],[N]):
  [x\otimes\xi] \to [ Mx\otimes \xi N], \,\,\,\forall [M],[N]\in\PGL(n+1,q) \]
and this action extends to all of $\PG(M_{n+1}(q))$ by linearity.
 The pure tensors $x\otimes \xi\in V\otimes V^*$ with $0\not= x\in V$ and $0\not=\xi\in V^*$  yield the matrices of $M_{n+1}(q)$ of rank $1$. With $x$ and $\xi$ as above, let $[x]$ and $[\xi]$ be the point and the hyperplane of $\PG(n,q)$ represented by $x$ and $\xi$. Then the point set of $\Gamma$ is precisely the set $\{ ([x], [\xi]), [x]\in \PG(V), [\xi]\in \PG(V^*)\}$ and
the image of $\Gamma$ under the Segre embedding is
\begin{equation}
  \label{eS}
\Lambda:=\varepsilon (\Gamma)=\{[x\otimes \xi]\colon  [x]\in\PG(V), [\xi]\in\PG(V^*)\},
\end{equation}
 also called the {\it Segre variety of $\PG((n+1)^2-1,q)$}. Accordingly, $\Lambda$ is represented by the set of all $(n+1)$-square matrices of rank $1$.

Regarding $\Lambda$ as a projective system of $\PG(M_{n+1}(q))$, we can consider the linear code $\cC(\Lambda)$ defined by it. This is  called the {\it Segre code}. It is easy to determine the length $N$ and the dimension $k$ of it. Indeed, $N$ is the number of point-hyperplane pairs of $\PG(n,q)$ and $k$ is the dimension of the Segre embedding. The minimum distance as well as the full weight enumerator are also known.
\begin{prop}[See \cite{BGH}]
   $\cC(\Lambda)$ is an $[N ,k ,d]$-code with
   \[N=\frac{(q^{n+1}-1)}{(q-1)} \frac{(q^{n+1}-1)}{(q-1)},\qquad k=(n+1)^2,
     \qquad d=q^{2n}.\]
\end{prop}

\subsection{The point-hyperplane geometry of a projective space.}
\label{long root geometry}
Let $\bar{\Gamma}$ be the subgeometry  of $\Gamma$ having as points, all the points  $(p,H)$ of $\Gamma$ with the further requirement that $p\in H$.
Two points are collinear in $\bar{\Gamma}$ if and only if they are
collinear in $\Gamma$; explicitly,
the points $(p,H)$ and $(p',H')$ are collinear in $\bar{\Gamma}$ if and only if $p\in H'$ or $p'\in H$.

If $p:=[x]$ and $H:=[\xi]$, then the point $([x], [\xi])$ of $\Gamma$ is a point of  $\bar{\Gamma}$ if and only if $\xi(x)=0$.
The geometry
$\bar{\Gamma}$ is called the
\emph{point-hyperplane geometry of $\PG(V)$} or, also, the
\emph{long root geometry for the special linear group} $\mathrm{SL}(n+1,\FF_q)$. The linear automorphism group of $\bar{\Gamma}$ is
$\PGL(n+1,q)$ and it acts transitively on the points of $\bar{\Gamma}$.

The group $\PGL(n+1,q)$ lifts  through the Segre embedding (see Definition~\eqref{segre embedding}) $\vep$ to a subgroup of the automorphism group
$\PGL(n^2+2n,q)$ of $\PG(M_{n+1}(q))$. In particular, it acts on the
pure tensors as follows
\begin{equation}\label{coniugio}
  g([x],[\xi])\,\,\stackrel{\varepsilon}{\longrightarrow}\,\, [ gx \otimes \xi g^{-1}],\,\,\forall g\in \PGL(n+1,q)
\end{equation}
and this action can be extended to all elements of $\PG(M_{n+1}(q))$ by linearity.
Consequently, $\PGL(n+1,q)$ acts on the
matrix representatives of the elements of $\PG(M_{n+1}(q))$ by conjugation and
its projective orbits correspond to the conjugacy classes of matrices
up to a non-zero proportionality coefficient.
The following is trivial.

\begin{prop}\label{prop rank 1}
 Suppose $x\in V$ and $\xi \in V^*$.  The vector $x\otimes \xi\in V\otimes V^*$, regarded as an $(n+1)$-square matrix of rank $1$,  is null-traced if and only if $\xi(x)=0$. 
\end{prop}
Denote by $M^0_{n+1}(q)$ the hyperplane of $M_{n+1}(q)$ of null-traced square matrices of order $n+1$, then  by Proposition~\ref{prop rank 1} we have $\vep^{-1}(M^0_{n+1}(q))=\bar{\Gamma}$.
So, according to Section~\ref{prelim}, $\bar{\Gamma}$ is a geometric hyperplane of $\Gamma$ arising from  $\vep$;
in the terminology of~\cite[Lemma 3]{Hvm24}, $\bar{\Gamma}$ is
a so-called \emph{black hyperplane} of $\Gamma$;
all black hyperplanes of $\Gamma$ correspond
to subgeometries of $\Gamma$ isomorphic to $\bar{\Gamma}$, but
these hyperplanes do not arise, in general, from $\vep$.
We point out that $M^0_{n+1}(q)$ is also the module which hosts
the {\it adjoint representation}\/ of the special linear group
$\mathrm{SL}(n+1,\FF_q)$.
\medskip

Let now $\bar{\vep}$ be the projective embedding of $\bar{\Gamma}$ obtained as the restriction ${\vep|}_{\bar{\Gamma}}$ of $ \varepsilon$ to $\bar{\Gamma}$.
 From Proposition~\ref{prop rank 1} we have
 \begin{equation}
   \label{star-1}
   \bar{\vep} \colon \bar{\Gamma} \rightarrow \PG(M^0_{n+1}(q)),\,\,\bar{\vep} (([x], [\xi]))=[x\otimes \xi].
   \end{equation}
  This is a projective embedding of $\bar{\Gamma}$ with dimension $\dim(\bar{\vep})=\dim(M^0_{n+1}(q))=(n+1)^2-1$.
  Define
\begin{equation}
  \label{e5}
\Lambda_1:=\bar{\vep} (\bar{\Gamma})=\{[x\otimes \xi]\colon  [x]\in \PG(V), [\xi]\in \PG(V^*) \,\,{\rm and}\,\,[x]\in [\xi]\}.
\end{equation}

Suppose now that $\FF_q$ admits non-trivial automorphisms and take  $\sigma\in\Aut(\FF_q),$  $\sigma\not=1$.
It is possible to define a \lq twisted version\rq\ $\bar{\vep}_{\sigma}$ of $\bar{\vep}$ as follows (see~\cite{DSSHVM} and \cite{Pasini24}):
\[\bar{\vep}_{\sigma} \colon \Gamma \rightarrow \PG(M_{n+1}(q)),\,\,\bar{\vep}_{\sigma} (([x], [\xi]))=[x^{\sigma}\otimes \xi],\]
where $x^{\sigma}:=({x_i}^{\sigma})_{i=1}^{n+1}$.
The map $\bar{\vep}_{\sigma}$ is again a projective embedding of $\bar{\Gamma}$ with dimension  $\dim(\bar{\vep}_{\sigma})=(n+1)^2$.
Put now
\begin{equation}
  \label{e7}
\Lambda_{\sigma}:=\bar{\vep}_{\sigma}(\bar{\Gamma})=\{[x^\sigma\otimes \xi]\colon  [x]\in\PG(V), [\xi]\in\PG(V^*) \text{ and } [x]\in [\xi]\}.
\end{equation}

In this paper we shall study in detail the properties of the projective
system $\Lambda_1$ since we want to construct a linear code from it.  In the forthcoming paper~\cite{parte2}, which is a direct continuation of the present work, we shall focus on $\Lambda_{\sigma}$ and on the code arising from it.
As we will soon illustrate, many properties of the geometry $\bar{\Gamma}$  play a crucial role in providing information on the structure of the code arising from $\Lambda_{1}$.
In particular, some families of geometric hyperplanes arising from $\bar{\vep}$ will be used to characterize in geometrical terms all the codewords  having minimum weight, second lowest weight and some
codewords having maximum weight.

\subsection{The saturation form}\label{forma saturata}
Let $f \colon  M_{n+1}(q) \times  M_{n+1}(q) \rightarrow \FF_q$ be the
non-degenerate symmetric bilinear form defined as
\begin{equation}\label{sat form}
  f(X, Y ) = \Tr(XY), \,\,\, \forall\,\, X, Y \in  M_{n+1}(q),
\end{equation}
where $XY$ is the usual row-times-column product
and $\Tr(XY)$ is the trace of the matrix $XY$.
So, with $X = (x_{i,j} )_{i,j=1}^{n+1}$ and
$Y = (y_{i,j} )_{i,j=1}^{n+1}$
we have
\[f((x_{i,j} )_{i,j=1}^{n+1}, (y_{i,j} )_{i,j=1}^{n+1}) =\sum_{i,j}x_{i,j}y_{j,i}.\]
Note that this definition does not depend on the choice of the basis of $M_{n+1}(q)$.
The form $f$ is called \emph{the saturation form} of $M_{n+1}(q)$ (see~\cite{Pasini24}).

Denote by $\perp_f$ the orthogonality relation associated to $f$. Since $f$ is nondegenerate, the hyperplanes of $M_{n+1}(q)$ are the orthogonal spaces  $M^{{\perp}_f}=\{X\in M_{n+1}(q)\colon \Tr(XM)=0 \},$ for $M \in  M_{n+1}(q) \setminus\{O\}$ and,
for two matrices $M,N \in M_{n+1}(q) \setminus\{O\}$, we have $M^{{\perp}_f} = N^{{\perp}_f}$ if and only
if $M$ and $N$ are proportional matrices.

By Definition~\eqref{sat form}, it is clear  that $I^{{\perp}_f}= M^0_{n+1}(q)$,
where $I$ is the identity matrix.
Therefore, for
$M \in M_{n+1}(q) \setminus \{O\}$, we have $M^{{\perp}_f} = M^0_{n+1}(q)$ if and only if $M=\lambda I$ with $\lambda\in\FF_q\setminus\{0\}$ i.e. $M$ is a
non-null scalar matrix.
Hence, every hyperplane of $M^0_{n+1}(q)$ can be obtained as $M^0_{n+1}(q)\cap M^{{\perp}_f}= I^{\perp_f}\cap M^{{\perp}_f}$ for a suitable matrix $M\in M_{n+1}(q),\,\, M\not\in \langle I \rangle$.
In terms of projective spaces, for $M\in M_{n+1}(q)\setminus\langle I\rangle$, $[M^{\perp_f} \cap M^0_{n+1}(q)]$ is the projective hyperplane of $[M^0_{n+1}(q)]$ corresponding to the hyperplane $M^{\perp_f} \cap M^0_{n+1}(q)$ of $M^0_{n+1}(q)$.

The next result follows from well-known properties of polarities associated to
non-degenerate reflexive bilinear forms; see~\cite[Proposition 2.1]{Pasini24}.
\begin{prop}\label{Prop perp}
  For $M,N \in M_{n+1}(q) \setminus  \{O\}$, we have $M^{{\perp}_f}\cap M^0_{n+1}(q) = N^{{\perp}_f}\cap M^0_{n+1}(q) $ if and only if $\langle M, I\rangle = \langle N, I\rangle$.
\end{prop}

The orthogonal space of a pure tensor, namely the orthogonal space of a matrix of rank $1,$  admits an easy description. Indeed,
\begin{prop}\label{prop}
Let $x \in V \setminus  \{O\}$, $\xi \in V^* \setminus  \{O\}$  and $M \in  M_{n+1}(q)$. Then
$x \otimes  \xi \in M^{\perp_f}$ if and only if $\xi M x = 0$.
\end{prop}
\begin{proof}
  By definition, $M^{\perp_f}=\{X\in M_{n+1}(q)\colon \Tr(XM)=0\}$. Hence $x \otimes  \xi \in M^{\perp_f}$ if and only if $\Tr((x\xi) M)=0$ if and only if $\Tr(x(\xi M))=0$. By Proposition~\ref{prop rank 1}, this is equivalent to $(\xi M)(x)=0$. Turning to projective spaces, this means that the projective point $[x]$ belongs to the hyperplane $[\xi M]$ if $M\not=0$ or that it (trivially)
 belongs to the space $\PG(V)=[0^*]$ if $M=0$. \end{proof}

In more geometrical terms, by Proposition~\ref{prop} we have that the point $[x \otimes  \xi]$ is contained in $[M^{\perp_f}]$ if and only if the point $[x]$ is contained in the hyperplane $[\xi M]$.

\subsection{Hyperplanes of $\bar{\Gamma}$}\label{hyperplanes}
In this section we will briefly recall from~\cite{Pas24} and~\cite{Pasini24} the most significant results related to the hyperplanes of $\bar{\Gamma}$ arising from the embedding $\bar{\vep}$.

Take $M\in M_{n+1}(q)\setminus\langle I\rangle$ and let $\bar{\vep}$ be the Segre embedding of $\bar{\Gamma},$ as defined in \eqref{star-1}. Then
\[{\cH}_M:=\bar{\vep}^{-1}([M^{\perp_f} \cap M^0_{n+1}(q)])
  =\bar{\vep}^{-1}(\{ [X]\in \PG(M^0_{n+1}(q))\colon \Tr(XM)=0 \})\]
is a geometric hyperplane of $\bar{\Gamma}$ called a \emph{hyperplane of plain type}, as defined in~\cite{Pasini24}.
By Proposition~\ref{Prop perp},  given any two matrices $M$ and $M'$ we have $\cH_M=\cH_{M'}$
if and only if $M=\alpha M'+\beta I$ with $(\alpha,\beta)\neq (0, 0)$.

Recall now the definition of hyperplanes of $\bar{\Gamma}$ arising from
an embedding from Section~\ref{prelim}.
\begin{prop}\cite[Corollary 1.7]{Pasini24}
  The hyperplanes of $\bar{\Gamma}$ which arise from the Segre embedding $\bar{\vep}$ are precisely those of plain type.
\end{prop}
Take $p\in\PG(V)$, $A\in\PG(V^*)$. Put ${\mathcal{M}}_p:=\{ (p, H) : p\in H \}$
and ${\mathcal{M}}_A:=\{ (x, A) : x\in A \}$. Then,
\begin{equation}\label{quasi-singular}
 \cH_{p,A}:=\{ (x,H) \colon (x,H) \text{ collinear (in $\bar{\Gamma}$) with a point of }
  \mathcal{M}_p\cup\mathcal{M}_A \}
  \end{equation}
is a geometric hyperplane of $\bar{\Gamma}$, called the \emph{quasi-singular hyperplane}
defined by $(p,A)$. If $p\in A$, then $\cH_{p,A}$ is called \emph{the singular
  hyperplane with deepest point $(p,A)$} and consists of all points of
$\bar{\Gamma}$ not at maximal distance from $(p,A)$ in the collinearity
graph of $\bar{\Gamma}$.

\begin{prop}\label{cardinalita' iperpiani quasi singolari}
The following hold.
\begin{enumerate}
	\item The cardinality of the singular hyperplanes of $\bar{\Gamma}$ is
          \begin{equation}\label{cc1}
            \frac{(q^{n+1}-1)(q^{n-1}-1)}{(q-1)^2}+\frac{q^n-1}{q-1} q^{n-1}.
          \end{equation}
	\item
	The cardinality of the quasi-singular but not singular hyperplanes of $\bar{\Gamma}$ is
	\begin{equation}\label{cc2}
          \frac{(q^{n+1}-1)(q^{n-1}-1)}{(q-1)^2}+(\frac{q^n-1}{q-1} +1)q^{n-1}.
        \end{equation}
\end{enumerate}	
\end{prop}
\begin{proof}
Suppose $\cH_{p,A}$
  with $p\in\PG(V)$ and  $A\in\PG(V^*)$ is a quasi-singular hyperplane  of $\bar{\Gamma}.$
  In order to determine the cardinality $|\cH_{p,A}|$ of $\cH_{p,A}$, we will first count the number $|C(\cH_{p,A})|$
  of points $(r,S)\in \bar{\Gamma}$ such that  $(r,S)\not\in \cH_{p,A}.$ By Definition~\ref{quasi-singular}, $|C(\cH_{p,A})|$ is precisely the number of points of $\bar{\Gamma}$ not collinear with any point in $\cM_p \cup \cM_A.$ Then, $|\cH_{p,A}|$ is the difference between the number of points of $\bar{\Gamma}$ and $|C(\cH_{p,A})|$, that is  \[ |\cH_{p,A}|=\frac{(q^{n+1}-1)}{q-1} \frac{(q^{n}-1)}{q-1} -|C(\cH_{p,A})|.\]

Suppose $p\not\in A$, i.e. $\cH_{p,A}$ is a quasi-singular, non singular hyperplane of $\bar{\Gamma}$. We have that $(r,S)$ is not collinear with any point in $\cM_p \cup \cM_A$ if and only if $r\not\in A$,
  $p\not\in S$ and $r\not=p.$
  More in detail, the number of points $r\in \PG(V)$ different from $p$ and not contained in $A$ is $\frac{(q^{n+1}-1)-(q^{n}-1)}{q-1} -1=q^n-1$ and the number of hyperplanes $S\in \PG(V^*)$ through the point $r$ and not containing $p$ is $\frac{(q^{n}-1)-(q^{n-1}-1)}{q-1}=q^{n-1}.$
  So, $|C(\cH_{p,A})|=q^{2n-1}-q^{n-1}.$

Now suppose $p\in A$, i.e. $\cH_{p,A}$ is a singular hyperplane of $\bar{\Gamma}$. We have that $(r,S)$ is not collinear with any point in $\cM_p \cup \cM_A$ if and only if $r\not\in A$ and   $p\not\in S.$
So, $|C(\cH_{p,A})|= (\frac{(q^{n+1}-1)}{q-1} - \frac{(q^{n}-1)}{q-1})(\frac{(q^{n}-1)}{q-1} - \frac{(q^{n-1}-1)}{q-1})  =q^{2n-1}.$
The claim follows.
  \end{proof}

The following theorem describes the quasi-singular hyperplanes
of $\bar{\Gamma}$.
\begin{prop}\cite[\S 1.3]{Pasini24}
\label{qsh}
  Take $[x]\in\PG(V)$ and $[\xi]\in\PG(V^*)$.
  The quasi-singular hyperplane $\cH_{[x],[\xi]}$
  is the  hyperplane of
  plain type $\cH_M$ where $M=x\xi$.
\end{prop}

By Proposition~\ref{qsh},  there is a one-to-one correspondence between
  quasi-singular hyperplanes of $\bar{\Gamma}$
  and proportionality classes of matrices of rank $1$.

In particular,
all quasi-singular hyperplanes are hyperplanes of plain type arising
from matrices $M$ of rank $1$ and, conversely, for each matrix $M\in M_{n+1}(q)$
of rank $1$ the hyperplane of plain type ${\cH}_M$ is quasi-singular.

Suppose $S$ is a \emph{line spread} of $\PG(V)$, that is a family of lines of $\PG(V)$ such that every point of $\PG(V)$ belongs to exactly one member of $S$.
We say that $S$ \emph{admits a dual} if there exists a line spread $S^*$ of
$\PG(V^*)$ such that for every line $\ell^*\in S^*$ (i.e.\ for every $2$-codimensional subspace of $\PG(V)$),  the members of $S$ contained in $\ell^*$,
form a line spread of $\ell^*$; see~\cite{Pasini24}.
A line spread $S$ admits at most one dual spread $S^*$, see~\cite[Lemma 1.9]{Pasini24}.
In~\cite{Pasini24} it is  proved that if a line spread $S$ admits a dual $S^*$, then 
it is possible to define a geometric hyperplane $\cH_{(S,S^*)}$ of $\bar{\Gamma}$ as follows
\begin{equation}\label{spread-type hyperplane}
  \cH_{(S,S^*)}:=\{(p,H)\in \bar{\Gamma}\colon\,\,H\supset \ell_p\}=\{(p,H)\in \bar{\Gamma}\colon\,\,p\in L_H\}
\end{equation}
where $\ell_p$ is the unique line of $S$ through $p$ and $L_H\in S^*$ is the unique $2$-codimensional subspace of $\PG(V)$ contained in $H$.
The hyperplane $\cH_{(S,S^*)}$ is called a {\it spread-type}\/ hyperplane of $\bar{\Gamma}$.

\begin{prop}~\cite[Theorem 1.14]{Pasini24}\label{spread-type 2}
 A hyperplane ${\cH}_M$ of plain type is of spread-type if and only if $M$ admits no eigenvalue in $\FF_q$ and $M^2 x\in \langle x,Mx\rangle$ for every non-zero vector $x\in V$.
\end{prop}

\section{The code $\cC(\Lambda_{1})$ from the Segre embedding}\label{long root geometry code 1}
In this section  we consider the subcode of the Segre code $\mathcal{C}(\Lambda)$
defined by the projective system $\Lambda_{1}\subset \PG(M^0_{n+1}(q))$; see Definition~\eqref{e5}.
We will denote by $\cC(\Lambda_1)$ the  $[N_1, k_1, d_1]$-linear code arising from $\Lambda_1$. 
The length of this code is the number of point-hyperplane pairs $(p,H)$ of $\PG(n,q)$ with $p\in H,$ that is
\[N_1=\frac{(q^{n+1}-1)(q^n-1)}{(q-1)^2}.\]
The dimension of $\cC(\Lambda_1)$ is the dimension of the embedding $\bar{\vep}$, so
\[k_1= n^2+2n.\]

To determine the weight of the codewords of $\cC(\Lambda_{1})$ we need to compute the cardinality of $\Lambda_{1}\cap [W]$ where $[W]$ is a hyperplane of $[\langle \Lambda_{1}\rangle]$.
Recall from Section~\ref{forma saturata} that any hyperplane of $\PG(V\otimes V^*)$
can be regarded as the orthogonal subspace $[M^{\perp_f}]$ of an $(n+1)\times (n+1)$-matrix $M$ with respect to the saturation form $f$.
The next lemma is crucial.
Observe that in this paper, when we speak of eigenvectors of
a matrix $M$ we always mean \emph{left} eigenvectors; also by
$\ker(M)$ we mean the set of row vectors $\xi$ such that $\xi M=0$.

\begin{definition}
  \label{thetaM}
  For any matrix $M\in M_{n+1}(q)$, denote by $\nu_M$ the number of
  eigenvectors of $M$ and by $\theta_M:=\frac{\nu_M}{q-1}$ the number of
  projective points of $\PG(V^*)$ whose row representatives are
  eigenvectors for $M$.
\end{definition}

\begin{lemma}\label{pesi}
   Let $[M^{\perp_f}]$ be a hyperplane of $\PG(V\otimes V^*),$ for $M\in M_{n+1}(q)\setminus\langle I\rangle$. Then
   \begin{equation}\label{Mnat} |\Lambda_{1}\cap[M^{\perp_f}]|=
     \frac{(q^{n+1}-1)(q^{n-1}-1)}{(q-1)^2}+
     \theta_M \cdot q^{n-1}
  \end{equation}
  where $\theta_M$ is given by Definition~\ref{thetaM}.
 \end{lemma}
 \begin{proof}
  First note that we can write
  $\Lambda_{1}=\{\bar{\vep}(([x],[\xi]))\colon [x]\in [\xi]\}$ as a disjoint union
  \[\Lambda_{1}=\bigsqcup_{[\xi]\in\PG(V^*)}\{ [x\otimes \xi] : [x]\in [\xi] \}.\]
  So,
   \[ [M^{\perp_f}]\cap\Lambda_{1}=\bigsqcup_{[\xi]\in\PG(V^*)}(\{ [x \otimes\xi] : [x]\in [\xi] \}\cap [M^{\perp_f}]). \]

   By Proposition~\ref{prop},
   $[x\otimes\xi]\in [M^{\perp_f}]$ if and only if $[x] \in [\xi M]$.
Hence, $[x \otimes \xi]\in \Lambda_{1}\cap [M^{\perp_f}]$ if and only if $[x] \in [\xi]$ and $[x]\in [\xi M],$ i.e.
   \[ \Lambda_1 \cap [M^{\perp_f}]=\bigsqcup_{[\xi]\in\PG(V^*)}(\{ [x \otimes \xi] : [x] \in([\xi] \cap [\xi M])\}). \]
 Turning to cardinalities,
 \begin{equation}
   \label{9bis} |\Lambda_1 \cap [M^{\perp_f}]|=\sum_{[\xi]\in\PG(V^*)}|[\xi] \cap[\xi M]|.
 \end{equation}
Note that $|[\xi]\cap[\xi M]|=|[\xi]|=\frac{(q^n-1)}{(q-1)}$ if $[\xi]\subseteq [\xi M]$ and  $|[\xi]\cap[\xi M]|=\frac{(q^{n-1}-1)}{(q-1)}$ otherwise. On the other hand, $[\xi]\cap[\xi M]=[\xi]$ if and only if $\xi M$ is a scalar multiple of $\xi$, that is $\xi$ is an eigenvector of $M$. Since $\theta_M$ denotes the number of points of $\PG(V^*)$
  whose vector representatives
  are eigenvectors of $M$ we have
\begin{multline*}
  |\Lambda_{1}\cap [M^{\perp_f}]|= \theta_M\cdot\frac{q^{n}-1}{q-1}+
  \left(\frac{q^{n+1}-1}{q-1}-\theta_M\right)\cdot\frac{q^{n-1}-1}{q-1}=\\
  \frac{(q^{n+1}-1)(q^{n-1}-1)}{(q-1)^2}+\theta_M\cdot q^{n-1}.
\end{multline*}
\end{proof}

\subsection{The codewords of $\cC(\Lambda_1)$}
Suppose $\Lambda_1:=\{[X_1],[X_2],\dots, [X_N]\}\subseteq \PG({M}_{n+1}^0(q))$ and denote by ${M}_{n+1}^*(q)$ the dual of the vector space
  $M_{n+1}(q)$.
  For any functional ${\mathfrak m}\in {M}_{n+1}^*(q)$,
  there exists a unique matrix $M\in M_{n+1}(q)$ such that
  ${\mathfrak m}={\mathfrak m}_M$ and
  \[{\mathfrak m}\colon {M}_{n+1}(q)\rightarrow \FF_q,\quad \mathfrak{m}(X)=\Tr(XM)\]
  for all $X\in M_{n+1}(q)$.
  Consider now the $N$-tuple
  \begin{equation}
    c_{\mathfrak{m}}=(\mathfrak{m} (X_1),\dots,\mathfrak{m}(X_N))
    \label{c_m}
  \end{equation}
  with     \begin{equation}\label{m_m} \mathfrak{m}(X_i)=\Tr(X_iM),\,\, 1\leq i\leq N, \end{equation}
  where $M\in {M}_{n+1}(q)$ is associated to $\mathfrak{m}$ as before.
  In this setting,
  \[\cC(\Lambda_1)=\{ c_{\mathfrak m}: \mathfrak{m}\in M_{n+1}^*(q) \}.\]
  In general there are  more than one functional $\mathfrak{m}$ (resp.\ matrix $M$) defining one codeword
  $c_{\mathfrak m}$.

\begin{lemma}
  \label{tlll}
  Let $M$ be a matrix such that  $\Tr(XM)=0$ for all $X\in M_{n+1}^0(q)$.
  Then $M\in\langle I\rangle$.
\end{lemma}
\begin{proof}
The claim is straightforward because $M_{n+1}^0(q)^{\perp_f}=(I^{\perp_f})^{\perp_f}=\langle I\rangle$.
\end{proof}

\begin{prop}\label{prop nuova}
  \label{oorb}
  $\cC(\Lambda_1)$ is vectorially isomorphic to the quotient space $M_{n+1}(q)/\langle I\rangle$.
\end{prop}
\begin{proof}
Define the evaluation function
\[ev:{M}_{n+1}^*(q)\to\cC (\Lambda_1),\qquad\mathfrak{m}\mapsto c_{\mathfrak{m}}.\]
By~\cite{TVZ}, $ev$ is linear and
surjective  and
$\cC(\Lambda_1)=\{   c_{\mathfrak{m}}\colon {\mathfrak m}\in {M}_{n+1}^*(q)\}=
ev(M_{n+1}(q))$.
The kernel of $ev$
is given by all $\mathfrak{m}$ such that $c_{\mathfrak{m}}=0$, i.e., by~\eqref{m_m}, the kernel of $ev$ can be identified with the space of all $M\in {M}_{n+1}(q)$ such that  $\Tr(XM)=0$ for all $X\in {M}_{n+1}^0(q),$ since $\langle \Lambda_1\rangle=\PG({M}_{n+1}^0(q))$.

By Lemma~\ref{tlll}, $\ker(ev)=\{ \mathfrak{m}_{\alpha I}\colon \alpha\in\FF_q\}$, where $\mathfrak{m}_{\alpha I}\colon {M}_{n+1}(q)\rightarrow \FF_q, \mathfrak{m}_{\alpha I}(X)=\Tr(X \alpha I)=\alpha\Tr(X)$.
So, the function $ev$ induces the vector space isomorphism
\[M_{n+1}^*(q)/\langle \mathfrak{m}_I\rangle \cong \cC(\Lambda_1).\]

Since the vector space $M_{n+1}(q)$ is isomorphic to $M_{n+1}^*(q)$,  we have that
\[ M_{n+1}(q)/\langle I\rangle\cong M_{n+1}^*(q)/\langle \mathfrak{m}_I\rangle
  \] and the claim follows.
\end{proof}

\begin{corollary}\label{co nuovo}
Suppose $\Lambda_1=\{[X_i]\colon i=1,\dots, N\}.$
\begin{enumerate}
\item If $p\nmid(n+1)$, then
\begin{multline*}
\cC(\Lambda_1)  =\{   c_{\mathfrak{m}}\colon {\mathfrak m}(X)=\Tr(XM)\,\, {\textrm with}\,\, M\in M_{n+1}^0(q)\}\\
 =\{  (\Tr(X_1M),\dots, \Tr(X_N M))\colon [X_i]\in \Lambda_1,\, M\in M_{n+1}^0(q)\}.
\end{multline*}
\item If $p|(n+1)$, then
\begin{multline*}
\cC(\Lambda_1)=
 \{  (\Tr(X_1M),\dots, \Tr(X_N M))\colon [X_i]\in \Lambda_1, \\
   M\in M_{n+1}(q) \text{ with } m_{1,1}=0\}.
\end{multline*}
\end{enumerate}

\end{corollary}
\begin{proof}
By Proposition~\ref{prop nuova}, the function
  $ev/\langle I\rangle \colon M_{n+1}(q)/\langle I\rangle\to \mathcal{C}(\Lambda_1)$
    is a vector space isomorphism.
Suppose $p\nmid (n+1)$. Each class $[\![M]\!]=M +\langle I\rangle $ in $M_{n+1}(q)/\langle I\rangle$ contains exactly one matrix $M_0\in M_{n+1}(q)$ of trace $0$; we choose that matrix as a
canonical\footnote{In more formal terms, the map $\pi:M_{n+1}/\langle I\rangle \to M_{n+1}^0(q)$ given
  by $\pi([\![M]\!]):=M-\Tr(M)I$ is a well defined vector space
  isomorphism, which commutes with matrix conjugation, in the sense
  that for all $g\in\GL(n+1,q)$, $\pi([\![M]\!]^g)=(\pi([\![M]\!]))^g$.}
  representative for the coset in $M_{n+1}(q)/\langle I\rangle$.

If $p\mid (n+1)$, then all matrices $M$  in the same class $[\![M]\!]=M +\langle I\rangle$
have the same trace. We can now choose as representative of $[\![M]\!]$ the
only matrix $N\in [\![M]\!]$ given by $N:=M-m_{1,1}I$ whose entry in
position $(1,1)$ is $0$.
\end{proof}

As a consequence of Proposition~\ref{prop nuova} and Theorem~\ref{main thm 2},
it is easy to define an efficient encoding function for $\cC(\Lambda_1)$, without
the need of explicitly writing out a generator matrix; in
particular if $[\![M]\!]\in M_{n+1}(q)/\langle I\rangle$ and
$X_1,X_2,\dots,X_N$ are matrix representatives of the points $[X_1],\dots,[X_N]$
of the projective system of $\Lambda_1$, then the codeword corresponding
to $[\![M]\!]$ is given by $(\Tr(X_1M),\dots,\Tr(X_NM))$.

\medskip

 The following is straightforward from Lemma~\ref{pesi}, considering that for any codeword $c\in \cC(\Lambda_{1})$, the weight of $c$ is $wt(c)=N_{1}-|[M^{\perp_f}]\cap\Lambda_{1}|$, where $M$ is the matrix associated to $c$.
\begin{corollary}\label{weights}
  The spectrum of weights of $\cC(\Lambda_{1})$	is
\[\{q^{n-1}\frac{(q^{n+1}-1)}{(q-1)} -q^{n-1} \theta_M\colon M\in M_{n+1}(q)\}\]
where $\theta_M$ is given by Definition~\ref{thetaM}.
\end{corollary}
\subsection{Proof of Theorem~\ref{main thm 1}}
\begin{lemma}\label{autovalori}
If $M$ is a diagonalizable matrix having $t>2$ eigenspaces of dimensions $g_1\geq g_2,\dots\geq g_t$ then it is always possible to construct a matrix $M'$ with $t-1$ eigenspaces of dimension respectively  $g_1'\geq g_2'\geq\dots\geq g_{t-1}'$, with $g_1'=g_1+g_2,$ $g_i'=g_{i+1}$, $2\leq i\leq t-1$ so that  $\nu_{M'}>\nu_M$.
   \end{lemma}
 \begin{proof}
The number of eigenvectors of $M$ is $\nu_M=\sum_{i=1}^t (q^{g_i}-1)$
   with $g_i\leq n$.   For $i=1,\dots t$, let $\lambda_i$ be the eigenvalue of $M$ corresponding to
   the eigenspace having dimension $g_i$. Define as follows a diagonal matrix $M'$ which has  $\lambda_2,\dots, \lambda_t$ as eigenvalues:
\[M':=diag(\underbrace{\lambda_2,\dots, \lambda_2}_{g_1+g_2},\underbrace{\lambda_3,\dots, \lambda_3}_{g_3}, \dots, \underbrace{\lambda_t,\dots, \lambda_t}_{g_t}).\]
Clearly, the dimensions of the eigenspaces of $M'$ are $g_1'\geq g_2'\geq\dots\geq g_{t-1}'$, with $g_1'=g_1+g_2,$ $g_i'=g_{i+1}$, $2\leq i\leq t-1$ and $\nu_{M'}=\sum_{i=1}^{t-1} (q^{g'_i}-1)= (q^{g_1+g_2}-1)+\sum_{i=3}^{t-1} (q^{g_i}-1)$.
  We have    $\nu_{M'}-\nu_M=(q^{g_1+g_2}-1)-(q^{g_1}-1)-(q^{g_2}-1)>0$    if and only if
   $q^{g_1+g_2}-q^{g_2}=q^{g_2}(q^{g_1}-1)> q^{g_1}-1$, that is $q^{g_2}>1$.
   Since $q>1$ and $g_2>0$, it follows $\nu_{M'}>\nu_M$.
 \end{proof}

 \begin{lemma}\label{max int}
  A non-scalar $(n+1)$-square matrix has a maximum number $\nu_{max}=q^n+q-2$ of eigenvectors
  if and only if it admits exactly two eigenspaces of
  respective dimensions $n$ and $1$.
 \end{lemma}
 \begin{proof}
   Let $M$ be a non scalar $(n+1)$-square matrix.
   If $M$ cannot be diagonalized, then the sum of the dimensions of its
   eigenspaces is at most $n$ and, consequently it has at most $q^n-1$
   eigenvectors.
   Suppose now that $M$ can be diagonalized and
   let $t$ be the number of eigenspaces of $M$ and
   $g_1\geq g_2\geq \dots\geq g_t$ be the respective dimensions of the
   eigenspaces. Since $M\notin \langle I\rangle$, we have $t\geq 2$.
   By recursively applying Lemma~\ref{autovalori}, we see that the maximum number
   of eigenvectors for a non-scalar diagonalizable $(n+1)$-square matrix can be attained only for $t=2$. So, suppose we have a matrix $M$ with just two eigenspaces of dimensions $g\leq n$ and $n+1-g$. Assume $g\geq n+1-g$. The number of eigenvectors of $M$ is then $\nu_M=(q^{g}-1)+(q^{n+1-g}-1)$.  This is maximum when $g=n$ and gives $\nu_{\max}=q^n+q-2\geq q^n>q^n-1$.
   The converse follows immediately.
 \end{proof}

 \medskip

  The length and dimension of $\cC(\Lambda_1)$ are computed  at the beginning of Section~\ref{long root geometry code 1}.

 By Corollary~\ref{weights}
 and Lemma~\ref{max int}, since $\theta_{max}=\nu_{max}/(q-1),$ we have that
the minimum distance of $\cC(\Lambda_1)$ is
\[d_1=
  q^{n-1}\frac{(q^{n+1}-1)}{(q-1)} -q^{n-1} \frac{q^n+q-2}{q-1}
  =q^{2n-1} -q^{n-1}.\]

Theorem~\ref{main thm 1} is proved.\hfill $\square$

\subsection{Proof of Theorem~\ref{main thm 2}.}
Let $M$ be a non-scalar matrix of $M_{n+1}(q)$ and let
$t$ denote the number of eigenspaces of $M$. Since the number of eigenspaces of any matrix is, clearly, the same as the number of its eigenvalues (which ranges in $\FF_q$), we have $t\leq q$ and since $M$ is non-scalar, $t\leq n$; so $t\leq\min(n,q)$.

Consider the following sets, where $\theta_M$ is given by
Definition~\ref{thetaM}:
\begin{equation}\label{autovettori}
	E=\left\{ \theta_M \colon M\in M_{n+1}(q)\setminus \langle I\rangle\right\}
      \end{equation}
    { and }
\begin{multline}\label{lista-}
D= \{0\}\cup\{(g_1,\dots, g_t)\colon
\sum_{i=1}^tg_i \leq n+1,\,\, \\ 1\leq g_1\leq \dots \leq g_i\leq g_{i+1}\leq \dots \leq  g_t\leq  n+1,
1 \leq t \leq q\}.
\end{multline}	
Observe first that if $m(x)$ is a monic irreducible polynomial over $\FF_q$
of degree $n+1$, then its companion matrix has order $n+1$
and does not have any eigenvalue in $\FF_q$. Clearly, such
a companion matrix is not a scalar matrix; so $0\in E$.

Assume $1\leq t\leq q$. Take $t$ distinct elements $\lambda_1,\dots,\lambda_t\in\FF_q$
 and for any $t$-tuple  $(g_1,\dots, g_t)\in D$, define a matrix $M_{(g_1,\dots, g_t)}\in M_{n+1}(q)\setminus \langle I \rangle$ as a block
matrix of the form
\[ M_{(g_1,\dots,g_t)}=diag(\underbrace{\lambda_1,\dots,\lambda_1}_{g_1},
  \underbrace{\lambda_2,\dots,\lambda_2}_{g_2},
  \dots
  \underbrace{\lambda_t,\dots,\lambda_t}_{g_t-1},
  R_{\lambda_t}) \]
where $R_{\lambda_t}$ is the Jordan block of order $n+2-\sum_{i=1}^t{g_i}$
of the form
\[ R_{\lambda_t}=\begin{pmatrix}
  \lambda_t & 1 & 0 & \dots & 0 \\
  0     &\lambda_t & 1 & \dots & 0 \\
  \vdots & &  \ddots & & \vdots \\
  0      & 0 & \dots  &  \lambda_t & 1 \\
  0      & 0 & \dots  &     0      & \lambda_t
  \end{pmatrix}. \]
Then $M_{(g_1,\dots,g_t)}$ has exactly
$t$ eigenspaces $V_1,\dots, V_t$ of respective dimensions
$g_1,\dots, g_t$ and  the number of its eigenvectors
is $\nu_{(g_1,\dots, g_t)}=\sum_{i=1}^t (q^{g_i}-1)$.
Since $\theta_{(g_1,\dots, g_t)}:=\nu_{(g_1,\dots, g_t)}/(q-1)$ we see that
the map
\begin{equation}\label{corrispondenza biettiva}
\begin{array}{l}
\varphi\colon D\rightarrow E\\
\varphi((g_1,\dots, g_t))=\theta_{(g_1,\dots, g_t)}\\
\varphi(0)=0.
\end{array}	
\end{equation}	
is well defined.
\begin{lemma}\label{injection}
	Take $1\leq t,t'\leq q$ and let $1\leq\alpha_1\leq\dots\leq\alpha_t$
	and $1\leq\beta_1\leq\dots\leq\beta_{t'}$ be integers such that
	\begin{equation}\label{eigeq} \sum_{i=1}^t (q^{\alpha_i}-1)=
		\sum_{i=1}^{t'} (q^{\beta_i}-1).
	\end{equation}
	Then $t=t'$ and $\alpha_i=\beta_i$ for all $i=1,\dots,t$.
\end{lemma}
\begin{proof}
  Reducing~\eqref{eigeq} modulus $q$ we obtain
	$-t\equiv -t'\pmod q$. As $1\leq t,t'\leq q$ this implies
	$t=t'$.
        Since $t=t'$, we can now rewrite~\eqref{eigeq} as
	\begin{equation}\label{eqr} \sum_{i=1}^t q^{\alpha_i}=
		\sum_{i=1}^{t} q^{\beta_i}.
	\end{equation}
	In order to prove the claim, we now proceed by induction on
        the number $t$ of terms in~\eqref{eqr}.
        \begin{itemize}
        \item
	If $t=2$, suppose $q^{\alpha_1}+q^{\alpha_2}=q^{\beta_1}+q^{\beta_2}$,
	that is
	\[ q^{\alpha_1}(1+q^{\alpha_2-\alpha_1})=q^{\beta_1}(1+q^{\beta_2-\beta_1}).
	\]
	Assume by contradiction
        $\alpha_1\neq\beta_1$; so we can take without loss of generality
	$\alpha_1>\beta_1$. So we get
	\begin{equation}\label{eq2aa}
	q^{\alpha_1-\beta_1}(1+q^{\alpha_2-\alpha_1})=(1+q^{\beta_2-\beta_1}),
      \end{equation}
      where all the exponents are non-negative.
	If we reduce~\eqref{eq2aa} modulus
        $q$ we get $(1+q^{\beta_2-\beta_1})\equiv0\pmod q$.
        If $q^{\beta_2-\beta_1}\neq 1$ this gives $1\equiv0\pmod q$,
        a contradiction. So it must be $\beta_2=\beta_1$ and we get
        $2\equiv0\pmod q$, which is possible only if $q=2$.
        However in this case~\eqref{eq2aa} becomes
	\[ 2^{\alpha_1-\beta_1}(1+2^{\alpha_2-\alpha_1})=2, \]
	which forces $\alpha_2=\alpha_1$
        and $\alpha_1=\beta_1$, contradicting the hypothesis.
        So $\alpha_1=\beta_1$ and, consequently, $\alpha_2=\beta_2$.
      \item
        Suppose $2<t<q-1$  By induction hypothesis, the condition
	\[ \sum_{i=1}^t q^{\alpha_i}= \sum_{i=1}^t q^{\beta_i}
	\Leftrightarrow (\alpha_1,\dots,\alpha_t)=(\beta_1,\dots,\beta_t) \]
	holds. We claim that it also holds for $t+1\leq q$ terms.
	If $\sum_{i=1}^{t+1} q^{\alpha_i}= \sum_{i=1}^{t+1} q^{\beta_i}$
	with $\alpha_{t+1}=\beta_{t+1}$, then,
        subtracting on the left and right hand side $q^{\alpha_{t+1}}$
        and then applying
        the inductive hypothesis we get
	$(\alpha_1,\dots,\alpha_t)=(\beta_1,\dots,\beta_t)$
        and we are done.
	Suppose then $\alpha_{t+1}\neq\beta_{t+1}$ and assume without loss
	of generality $\beta_{t+1}\leq\alpha_{t+1}-1$ and that,
        by contradiction,
        \[ \sum_{i=1}^{t+1} q^{\alpha_i}= \sum_{i=1}^{t+1} q^{\beta_i}.
        \]
        Observe that for all $1\leq i\leq t+1$
        we have $q^{\beta_i}\leq q^{\beta_{t+1}}\leq
        q^{\alpha_{t+1}-1}$ and $q^{\alpha_i}\geq q^0=1$,
        so
        \begin{equation*} q^{\alpha_{t+1}}=
		\sum_{i=1}^{t+1}q^{\beta_i}-\sum_{i=1}^{t}q^{\alpha_i}
		\leq 
		\sum_{i=1}^{t+1}q^{\alpha_{t+1}-1}-\sum_{i=1}^t1=
		(t+1)q^{\alpha_{t+1}-1}-t.
	\end{equation*}
	Since $0<t+1\leq q$ this implies $q^{\alpha_{t+1}}\leq q^{\alpha_{t+1}}-t$
	which is a contradiction.
	It follows that it must be $\alpha_{t+1}=\beta_{t+1}$.
	This completes the proof.
      \end{itemize}
\end{proof}

\begin{prop}\label{biezione}
The sets $E$ and $D$  defined in~\eqref{autovettori} and~\eqref{lista-} are in bijective correspondence.
\end{prop}
\begin{proof}
We need to prove that the map defined in~\eqref{corrispondenza biettiva} is a bijective correspondence.

Injectivity follows from Lemma~\ref{injection}.
For the surjectivity, $0\in D$ is uniquely mapped to $0\in E$. If $\theta\in E$ and $\theta\geq 1$, by definition of $D$ there exists
a matrix $M\in M_{n+q}(q)\setminus\langle I\rangle$ with
$\theta_M=\theta$ where
$\nu_M=(q-1)\theta_M$ is the number of its
eigenvectors. By Lemma~\ref{injection} the list $(g_1,\dots g_t)$ of the dimensions of the eigenspaces of $M$ is uniquely determined by $\nu_M$.
In particular $g_1,\dots,g_t$ must satisfy $1\leq t\leq q$ and
$\sum_{i=1}^t g_i\leq n+1$.
We can take without loss of generality $g_i\leq g_j$ if $i\leq j$.
So there is $(g_1,\dots,g_t)\in D$ such that
$\varphi(g_1,\dots,g_t)=\theta_M$.
\end{proof}
Part~\ref{t-cor}~of Theorem~\ref{main thm 2} follows from Proposition~\ref{biezione}
and Corollary~\ref{weights}.

We now prove Part~\ref{t-isom} of Theorem~\ref{main thm 2}.
Recall that an automorphism  of a code
$\cC(\Lambda)$ is a linear map $\cC(\Lambda)\to\cC(\Lambda)$
preserving the weights of all of the codewords; see~\cite{MS}.

\begin{prop}
\label{pp1}
The code $\cC(\Lambda_1)$ admits the
  group $\PGL(n+1,q)$ as an automorphism group acting
  transitively on the components
  of the codewords.
\end{prop}
\begin{proof}
  We need to distinguish two cases.
  \begin{enumerate}[(A)]
  \item\label{CA}
    \fbox{$p\nmid (n+1)$} In this case
 $M_{n+1}(q)/\langle I\rangle\cong M_{n+1}^0(q)$;
 see~\cite[Lemma 1.5]{Cardinali2025}
 Since the group $\GL(n+1,q)$ acts by conjugation
  on $M_{n+1}^0(q)$ and
  the kernel of this action is given by the scalar
  matrices $\langle I\rangle$,  the group of transformations
  induced by $\GL(n+1,q)$ on $M_{n+1}^0(q)$ is isomorphic to $\PGL(n+1,q)$. 
  By  Proposition~\ref{prop nuova} and the proof of Corollary~\ref{co nuovo}, $\cC(\Lambda_1)$ and $M_{n+1}^0(q)$ are isomorphic as vector spaces and for any codeword $c\in\cC(\Lambda_1)$
  there exists $M\in M_{n+1}^0(q)$
  such that $c=c_{\mathfrak{m}},$ where $\mathfrak{m}\in M_{n+1}^*(q)$ and
  $\mathfrak{m}(X)=\Tr(XM).$

  For $g\in\GL(n+1,q)$, consider the action $\phi$ on $\cC(\Lambda_1)$
  given by $c\to c^g$ where $c^g$ is
  $c^g=c_{{{\mathfrak m}}^g}$ with  ${\mathfrak m}^g(X):=\Tr(Xg^{-1}Mg)$.
  Clearly,   $\Tr(M)=\Tr(g^{-1}Mg)=0$; so $c^g\in \cC(\Lambda_1)$.

  The map $\phi$ is linear, since for any two codewords $c_1,c_2$
  induced respectively by matrices $M_1$ and $M_2$ with functionals
  $\mathfrak{m}_1$ and $\mathfrak{m}_2$ we have, for any $\alpha,\beta\in \FF_q$
  \begin{multline*}
    (\alpha c_1+\beta c_2)^g=((\alpha\mathfrak{m}_1+\beta\mathfrak{m}_2)^g(X_1),
    \dots(\alpha\mathfrak{m}_1+\beta\mathfrak{m}_2)^g(X_N))=\\
    (\Tr(X_1g^{-1}(\alpha M_1+\beta M_2)g),\dots,
    (\Tr(X_Ng^{-1}(\alpha M_1+\beta M_2)g))=\\
    (\Tr(\alpha X_1g^{-1}M_1g+\beta X_1g^{-1}M_2g),\dots,
    (\Tr(\alpha X_Ng^{-1}M_1g+\beta X_Ng^{-1}M_2g))=\\
    \alpha(\Tr(X_1g^{-1}M_1g),\dots,\Tr(X_Ng^{-1}M_1g))+\\
    \beta(\Tr(X_1g^{-1}M_2g),\dots,\Tr(X_Ng^{-1}M_2g))=\\
    \alpha(\mathfrak{m}_1^g(X_1),\dots,\mathfrak{m}_1^g(X_N))+
    \beta(\mathfrak{m}_2^g(X_1),\dots,\mathfrak{m}_2^g(X_N))=
    \alpha c_1^g+\beta c_2^g.
  \end{multline*}
  Since the number of eigenvectors of $M$ and that of $g^{-1}Mg$
  is the same and, by Corollary~\ref{weights}, the weight of a codeword depends on the number of eigenvectors defining it, the weight of $c^g$ is the same as that of $c$.
  Finally, since the group $\PGL(n+1,q)$ acts flag-transitively
 on
  the geometry $\PG(V)$, it is also transitive on the geometry
  $\bar{\Gamma}$ and, consequently,
  by the homogeneity of the embedding $\bar{\vep}$, also on the
  projective system $\Lambda_1$.
  This completes the proof of this case.

  \item \fbox{$p\mid (n+1)$}
  By Proposition~\ref{prop nuova}, $ev/\langle I\rangle \colon M_{n+1}(q)/\langle I\rangle\to
  \cC(\Lambda_1)$ is a vector space isomorphism but in
  this case we cannot identify $M_{n+1}(q)/\langle I\rangle$
  with $M^0_{n+1}(q)$.
  The group $\GL(n+1,q)$ acts on the space $M_{n+1}(q)$ by
  conjugation and since $\langle I\rangle$ is fixed (element-wise) by
  this action,  it also acts on $M_{n+1}(q)/\langle I\rangle$ by
  considering the action on the equivalence classes.
  Indeed for any $[\![M]\!]\in M_{n+1}(q)/\langle I\rangle$ we have
  $[\![M]\!]=\{ M+\lambda I: \lambda\in\FF_q\}$; as
  $g^{-1}(M+\lambda I)g=g^{-1}Mg+\lambda I$ for any $g\in\GL(n+1,q)$,
  we have that
  the map $[\![M]\!]\to [\![M]\!]^g:=[\![g^{-1}Mg]\!]$ is
  well defined.
  The kernel of this action is given by $\langle I\rangle$; so
  $\PGL(n+1,q)=\GL(n+1)/\langle I\rangle$ acts
  faithfully on $M_{n+1}(q)/\langle I\rangle$.
  As in Case~(\ref{CA}), this induces an action
  on the components $X_i$ of
  the codeword $c_{[\![M]\!]}=(\Tr(X_1M),\dots,\Tr(X_NM))$
    by $c_{[\![M]\!]} \to  c_{[\![M]\!]}^g:=c_{[\![M]\!]^g}$ and
    this action is linear. It remains to prove that the action
    is an isometry. However, the weight of the codeword
    $c_{[\![M]\!]}$ depends only on the number of eigenvectors
    of any matrix in ${[\![M]\!]}$ (and all such matrices have
    the same number of eigenvectors). As conjugate matrices have
    also the same number of eigenvectors, it follows that
    the weight of  $c_{[\![M]\!]}^g=c_{[\![M]\!]^g}$ is the same
    as the weight of $c_{[\![M]\!]}$. This completes the proof.
  \end{enumerate}
\end{proof}

\begin{remark}
  The automorphism group of a code over $\FF_q$ always
  contains a cyclic subgroup isomorphic to $\FF_q^*$ in its center,
  acting on the words by scalar multiplication.
  As $\PGL(n+1,q)$ has a trivial center, we can see that the automorphism
  group of $\cC(\Lambda)$ must be the direct product
  $\PGL(n+1,q)\cdot\FF_q^*$. Observe however that the action of this
  group is not the action of $\GL(n+1,q)$ by conjugation and that the action of $\PGL(n+1,q)$ on the codewords
  is different in the case $p\nmid(n+1)$ and $p\mid(n+1)$.
\end{remark}

\subsection{Proof of Theorem~\ref{main thm 2-bis}.}
In 2025, Pasini~\cite[Theorem 1.5]{Pasini24} proved that all hyperplanes
of the point-hyperplane geometry $\bar{\Gamma}$ of $\PG(V)$ are maximal subspaces; so point~\ref{mw-1} of Theorem~\ref{main thm 2-bis} follows
immediately from Proposition ~\ref{c:min}.

\begin{remark}
  We point out that the minimality of $\cC(\Lambda_1)$
  does not follow directly from the Ashikhmin and Barg
  Condition~\eqref{ab-bound} of Section~\ref{ssmin}. Indeed,
  in order to apply Condition~\eqref{ab-bound} to test a code for
  minimality we need to know the minimum and the maximum weight
  of the codewords. These are known or $\cC(\Lambda_1)$,
  thanks to Theorem~\ref{main thm 2} and Corollary~\ref{weights} and
  turn out to be
  $w_{\max}=q^{n-1}(q^{n+1}-1)/(q-1)$ and
  $w_{\min}=q^{n-1}(q^n-1)$.
  However,
  \[ \frac{w_{\max}}{w_{\min}}=\frac{(q^{n+1}-1)}{(q-1)(q^n-1)}>\frac{q}{q-1}.
    \]
\end{remark}

The following is an elementary remark on matrices of rank $1$.
\begin{lemma}
  \label{rank 1}
  A rank $1$ matrix $M\in M_{n+1}(q)$ is diagonalizable if and only if
  $\Tr(M)\neq 0$.
\end{lemma}
\begin{proof}
  Suppose $M$ is diagonalizable of rank $1$. Clearly,
  the eigenspace $V_0$ of $0$ must have dimension $n$. So there must also
  be another eigenspace $V_\lambda$ of dimension $1$ associated to
  a non-zero eigenvalue $\lambda$. The trace of the diagonalized matrix
  is thus $\lambda\neq0$ and, since similar matrices have the same trace,
  also $\Tr(M)=\lambda\neq0$.
  Suppose now $M$ has rank $1$ and it is not diagonalizable. Choose a basis $(e_1,\dots,e_{n+1})$ of $\FF_q^{n+1}$
  where the first $n$ vectors belong to the kernel of $M$.
  So, up to conjugation,
  we can assume without loss of generality that
  only the last column of $M$ contains non-zero entries.
  The characteristic polynomial $p_M(\lambda)$ of such a matrix $M$ is
  $(-\lambda)^n(m_{n+1,n+1}-\lambda)$. If it were $m_{n+1,n+1}\neq 0$,
  then $M$ would have one further eigenvalue with eigenspace of
  dimension (at least) $1$; so $M$ would be diagonalizable, a contradiction.
  It follows
  that it must be $m_{n+1,n+1}=0$; consequently, $\Tr(M)=0$.
\end{proof}

 \begin{prop}\label{min codewords}
 Any minimum weight codeword of $\cC(\Lambda_1)$ corresponds to a hyperplane of
   the form $[N^{\perp_f}]$ where $N$ is a matrix of rank $1$ and
   trace different from $0$.
 \end{prop}
 \begin{proof}
   By Lemma~\ref{max int}, any  minimum weight codeword corresponds to a hyperplane $[M^{\perp_f}]$ where $M$ is a diagonalizable non-scalar $(n+1)$-square matrix  having two eigenspaces of dimensions $n$ and $1$. Denote by $\lambda_1$ and $\lambda_2$ the two eigenvalues of $M$. By Proposition~\ref{Prop perp}  we have $[M^{\perp_f}]\cap \Lambda_1=[(M-\lambda_1 I)^{\perp_f}]\cap \Lambda_1$. Put $N:=M-\lambda_1 I$.
   So, by Proposition~\ref{oorb}, $N$ and $M$ determine the same codeword. Clearly, $N$ has eigenvalues $0=\lambda_1-\lambda_1$ and $\lambda_2-\lambda_1$
   with multiplicities $n$ and $1$.
   Thus, $N$
   is diagonalizable with rank $1$.
   By Lemma~\ref{rank 1}, this implies that $N$ has trace different from $0$.
 \end{proof}

\begin{prop}\label{second lowest}
Any second lowest weight codeword of $\cC(\Lambda_1)$ corresponds to a hyperplane of the form
$[N^{\perp_f}]$ where $N$ is a (non-diagonalizable) matrix of rank $1$ and trace $0$.
\end{prop}
\begin{proof}
	By Lemma~\ref{max int} and Proposition~\ref{min codewords} the minimum
        weight codewords occur when $M$
	has rank $1$ and trace different from $0$, i.e. $M$ has rank $1$ and is diagonalizable;
	in this case $\theta=\theta_{\min}=\frac{q^n-1}{q-1}+1$.
	Suppose now that $N$ is a  non-diagonalizable  matrix of rank $1$.
        By Lemma~\ref{rank 1}, this is the same to say
	that $N$ has rank $1$ and trace $0$. Then, $\theta_N=\frac{q^n-1}{q-1}=\theta_{\min}-1$. It follows from Corollary~\ref{weights} that
	$N$ determines a codeword with second lowest weight.

	Conversely, suppose that $N$ defines a codeword with
        the second lowest weight. By Corollary~\ref{weights},
        $\theta_N=\theta_{\min}-1=\frac{q^n-1}{q-1}$.
        By Proposition~\ref{oorb}, we
	can replace $N$ with $N-\lambda I$ where $\lambda$ is the eigenvalue of $N$ whose eigenspace has the highest
	dimension. If $N$ has rank $1$, i.e.\ just one eigenspace, then we are done. Otherwise, suppose that $N$ has $t>1$ eigenspaces
	and let $d_1=(n+1)-\rank(N)\geq d_2\geq\dots\geq d_t$ be the corresponding
	dimensions. Clearly, $d_1<n$. Counting the eigenvectors of $N$ we have that the following must
	hold
	$\nu_N=\sum_{i=1}^t (q^{d_i}-1)=q^n-1$, that is $\nu_n=q^{d_1}+q^{d_2}+\dots+q^{d_t}=q^n+t-1$.
	On the other hand
	\[
	 q^{d_1}+q^{d_2}+\dots+q^{d_t}\leq q^{d_1}+q^{d_1-1}+\dots+1=\frac{q^{d_1+1}-1}{q-1}\leq
		\frac{q^n-1}{q-1}<q^n-1+t.
	\]
	This is a contradiction. So, it must be $t=1$ and $N$ has rank $1$.
\end{proof}
Propositions~\ref{min codewords} and~\ref{second lowest} give a geometrical interpretation for the minimum and the second lowest
weight codewords. Note also that a matrix $N$ of rank $1$ determines a
minimum weight codeword of $\cC(\Lambda_1)$ if and only if $[N]\not\in\Lambda_1$ while it determines a codeword with the second lowest weight if and only if $[N]\in\Lambda_1$.
\\

 By Section~\ref{hyperplanes}, the quasi-singular hyperplanes of $\bar{\Gamma}$ are in correspondence with matrices of rank $1$. By Proposition~\ref{min codewords}, any minimum weight codeword corresponds to a matrix of rank $1$ and trace different from $0$. Hence any minimum weight codeword of $\cC(\Lambda_1)$ corresponds to a quasi-singular hyperplane $\cH:=\cH_{([x],[\xi])}$ of $\bar{\Gamma}$ defined by a matrix $x\otimes \xi$ with non-null trace. This last condition amounts to say that $[x]\not\in[\xi],$ i.e. $\cH$ is a quasi-singular but not singular hyperplane of $\bar{\Gamma}$.

Take now a codeword having second lowest weight. By Proposition~\ref{second lowest}, it is in correspondence with a quasi-singular hyperplane $\cH':=\cH_{([y],[\eta])}$ of $\bar{\Gamma}$ defined by a matrix $y\otimes \eta$ with null trace. This last condition amounts to say that $[y]\in [\eta],$ i.e. $\cH'$ is a singular hyperplane of $\bar{\Gamma}$. Part~\ref{mw0} of Theorem~\ref{main thm 2-bis} is proved.

We now focus on the maximum weight codewords. By Corollary~\ref{weights}
we immediately have
 \begin{prop}
\label{maxwpr}
   Maximum weight codewords have weight $q^{n-1}(q^{n+1}-1)/(q-1)$  and correspond to hyperplanes of
   the form $[M^{\perp_f}]$ where $M$ has no eigenvalues in $\FF_q$.
 \end{prop}
 By Proposition~\ref{spread-type 2},
 a hyperplane $\cH_M$ of $\bar{\Gamma}$ is of spread-type if and only if  $M$ is a matrix
 having no eigenvalue in $\FF_q$ and $M^2x\in\langle x, Mx\rangle$
 for any $x\in V\setminus\{0\}$. By Proposition~\ref{maxwpr},
 maximum weight codewords have weight $q^{n-1}(q^{n+1}-1)(q-1)$ achieved for $\theta_M=0$ and correspond to hyperplanes $[M^{\perp_f}]$ where $M$ has no eigenvalues in $\FF_q$. So, any  spread-type  hyperplane of $\bar{\Gamma}$ arising from the Segre embedding $\bar{\vep}$ of $\bar{\Gamma}$ is associated to a maximum weight codeword.

 The following proposition provides conditions for the converse.
 \begin{prop}
   \label{ppmw}
   Any maximum weight codeword of $\mathcal{C}(\Lambda_1)$ corresponding to a matrix $M\in M_{n+1}(q)\setminus\langle I\rangle$ such that the minimal polynomial of $M$ is irreducible of degree $2$
   is a spread-type hyperplane
   of $\bar{\Gamma}$.
 \end{prop}
 \begin{proof}
   Since $M$ is associated to a maximum weight codeword,
   by Proposition~\ref{maxwpr} $M$ has no eigenvalue in $\FF_q$.
   By hypothesis,
   the minimal polynomial  of $M$ is an irreducible polynomial over $\FF_q$
   of the form $p_{M}(x):=x^2+\alpha x+\beta,$ for $\alpha,\beta\in \FF_q$. Then $M^2+\alpha M+\beta I=0$. So, $M^2v=-\alpha Mv-\beta v, \forall v\in V$. Hence, $M^2v\in \langle Mv, v \rangle$ and,
 by Proposition~\ref{spread-type 2}, $\cH_M$ is a spread-type  hyperplane of $\bar{\Gamma}$.
\end{proof}
Theorem~\ref{main thm 2-bis} is proved. \par\hfill $\square$\par

\section*{Acknowledgments}
Both authors are affiliated with GNSAGA of INdAM (Italy) whose support they kindly acknowledge.
This work was partially supported by the project ``CONSTR: a COllectionless-based Neuro-Symbolic Theory for learning and Reasoning'', PARTENARIATO ESTESO ``Future Artificial Intelligence Research - FAIR'', SPOKE 1 ``Human-Centered AI'' Universit\`a di Pisa,  ``NextGenerationEU'', CUP I53C22001380006.



\end{document}